\documentclass[reqno]{amsart}
\usepackage{}
\usepackage[top=1in, bottom=1in, left=1in, right=1in]{geometry}
\usepackage{cases}
\usepackage[colorlinks,
            linkcolor=red,
            anchorcolor=blue,
            citecolor=green
            ]{hyperref}
\usepackage{amsmath,amsthm,amsfonts,amssymb,amscd,bm}
\usepackage{hyperref}
\newtheorem{theorem}{Theorem}[section]
\newtheorem{lemma}[theorem]{Lemma}
\theoremstyle{definition}
\newtheorem{definition}[theorem]{Definition}
\newtheorem{example}[theorem]{Example}

\newtheorem{prop}[theorem]{Proposition}

\newtheorem{cor}[theorem]{Corollary}

\newtheorem{algo}[theorem]{Algorithm}
\theoremstyle{remark}
\newtheorem{remark}[theorem]{Remark}
\renewcommand{\dim}{\mbox{\rm dim}}
\renewcommand{\char}{\mbox{\rm char}}

\newcommand{\GL}{\mbox{\rm GL}}

\newcommand{\Gal}{\mbox{\rm Gal}}
\numberwithin{equation}{section}

\begin{document}\large

\title{Computing points on modular curves over finite fields}

\author{Jinxiang Zeng}
\address{Department of Mathematical Science, Tsinghua University, Beijing 100084, P. R. China}
%\curraddr{Department of Mathematics and Statistics,
%Case Western Reserve University, Cleveland, Ohio 43403}
\email{cengjx09@mails.tsinghua.edu.cn}


%    General info
\subjclass[2012]{Primary 11F37, 11F30, 11G20, 11Y16, 14Q05, 14H05}
%11F37: Forms of half-integer weight; nonholomorphic modular forms
%11F30: Fourier coefficients of automorphic forms
%11G20: Curves over finite and local fields, See also 14H25
%11Y16: Algorithms; complexity, See also 68Q25
%14H05: Algebraic functions; function fields, See also 11R58

%\date{August 15, 2011 and, in revised form, June 22, 2001.}


\keywords{Modular forms, Hecke algebra, modular curves, elliptic curves, Jacobian, order counting}

\begin{abstract}
In this paper, we present a probabilistic algorithm to compute the number of $\mathbb{F}_p$-points of modular curve $X_1(n)$. Under the Generalized Riemann Hypothesis(GRH), the algorithm takes $\textrm{O}(n^{56+\delta+\epsilon}\log^{9+\epsilon} p)$ bit operations, where $\delta$ is an absolute constant and $\epsilon$ is any positive real number. As an application, we can compute $\#X_1(17)(\mathbb{F}_p)\textrm{ mod } 17$ for huge primes $p$. For example, we have $\#X_1(17)(\mathbb{F}_{10^{1000}+1357})\textrm{ mod } 17=3$.
\end{abstract}

\maketitle

\section{Introduction and Main Results}

Let $\mathcal{C}$ be a general curve of genus $g$ over finite field $\mathbb{F}_p$, how fast can we determine the cardinality of $\mathcal{C}(\mathbb{F}_p)$? For a survey of this problem, we refer to \cite{Poonen}. The algorithms known to date all have running time exponential in the genus $g$. For example, the number of $\mathbb{F}_q$-rational points on the Jacobian of a genus $g$ hyperelliptic curve can be computed in $\textrm{O}((\log q)^{\textrm{O}(g^2\log g)})$, proved by Adleman and Huang in \cite{Huang2001}. It would be very nice to have algorithms run polynomial both in $\log p$ and $g$. In the case of modular curve $X_1(n)$, it was proved by Edixhoven, Couveignes et al. in \cite{Edixhoven} and Bruin in \cite{Bruin}, assuming the Generalised Riemann Hypothesis (GRH), there exists an algorithm determines $X_1(n)(\mathbb{F}_p)$ in time polynomial in $n$ and $\log p$.

In this paper, we describe such an algorithm, give a very rough estimation of the complexity and some explicit calculations as well.  We only consider the case while the level $n$ is a prime number.

Let $P_n(t)$ be the characteristic polynomial of the Frobenius endomorphism $\textrm{Frob}_p$ on $J_1(n)$, then $P_n(t)\in\mathbb{Z}[t]$ is a monic polynomial of degree $2g$ with roots denoted as $\alpha_i,1\le i\le 2g$. We have
$$\# X_1(n)(\mathbb{F}_p)=1-\sum_{i=1}^{2g}\alpha_i+p,$$
and
$$\# J_1(n)(\mathbb{F}_p)=\prod_{i=1}^{2g}(1-\alpha_i)=P_n(1).$$
Thus we will concern ourselves  with the computation of the characteristic polynomial $P_n(t)$.

Let $f\in S_2(\Gamma_1(n))$ be a newform with q-expansion $f(q)=\sum_{i=1}^{\infty}a_n q^n$, $\mathbb{Q}_f=\mathbb{Q}[a_n(f):n\ge1]$ the number field generated by the coefficients of $f$, $\mathbb{Z}_f$ the ring of integers of $\mathbb{Q}_f$. Let $\mathcal{O}_f:=\mathbb{Z}[a_n(f):n\ge1]$, then we have that $\mathcal{O}_f\subset \mathbb{Z}_f$ is an order of $\mathbb{Q}_f$. For any prime ideal $\mathfrak{p}$ of $\mathcal{O}_f$, denote $\mathbb{F}=\mathcal{O}_f/\mathfrak{p}$, then we have a surjective ring homomorphism $\theta_f:\mathbb{T}(n,2)\to \mathbb{F}$ and a  mod-$\ell$ Galois representation
$$\rho_{f}:\textrm{Gal}(\overline{\mathbb{Q}}/\mathbb{Q})\to \textrm{GL}_2(\mathbb{F}),$$
where $\mathbb{T}(n,2)=\mathbb{Z}[T_n:n\ge 1]\subset \textrm{End}(S_2(\Gamma_1(n)))$ is the Hecke algebra. The representation space of $\rho_f$ is $J_1(n)[\mathfrak{m}]$, where $\mathfrak{m}=\ker\theta_f$ is a maximal ideal of the Hecke algebra $\mathbb{T}(n,2)$. The results in \cite{Edixhoven} and \cite{Bruin} show that: for prime $p\nmid n\ell$, the characteristic polynomial (denoted as $P_f(t)$) of $\rho_f(\textrm{Frob}_p)$, can be computed polynomial in $n$, $\log p$ and $\#\mathbb{F}$.

On the other hand, we have $P_n(t)\mod \ell= \prod_{f}P_f(t)$, where $f$ runs through a set of representatives for $\textrm{Gal}(\overline{\mathbb{F}}_\ell/\mathbb{F}_\ell)$-conjugate classes of newforms in $S_2(\Gamma_1(n),\overline{\mathbb{F}}_\ell)$.
%here we will use l splits completely, so don't worry about the multiplicity needed
Thus, essentially, the computation of $P_n(t)\mod \ell$ comes down to the computation of the mod-$\ell$ Galois representation associated to each newform $f$.

The heights of coefficients in $P_n(t)$ are well bounded. So we can recover $P_n(t)$ by the Chinese Remainder Theorem from $P_n(t)\mod \ell$ for sufficiently many small primes $\ell$.

{\bf Notation}: The running time will always be measured in bit operation. Using FFT, multiplication of two $m$-bit length integers can be done in $\textrm{O}(m^{1+\epsilon})$. Multiplication in finite field $\mathbb{F}_q$ can be done in $\textrm{O}(\log^{1+\epsilon} q)$. We denote $\epsilon$  any positive real number. Cardinality of a finite set $S$ is denoted by $|S|$ or $\# S$. For any newform $f$, we denote $\mathbb{Q}_f$ the number field generated by the coefficients of $f$, $\mathbb{Z}_f$ the ring of integers of $\mathbb{Q}_f$ and $\mathcal{O}_f$ the order of $\mathbb{Q}_f$ generated by the coefficients of $f$.

In order to get a view of the complexity of our algorithm,  we have to introduce two more absolute constants, denoted as $\omega$ and $\delta$. The constant $\omega$ refers to that: a single group operation in the Jacobian $J_1(n)$ can be done in time $\textrm{O}(g^\omega)$, where $g$ is the genus of the modular curve $X_1(n)$. The constant $\omega$ is in $[2,4]$ see \cite{KM} and \cite{Hess}.

The constant $\delta$ is used to measure the heights of elements of the two dimensional $\mathbb{F}$-vector space $J_1(n)[\mathfrak{m}]$. More precisely, let $\iota:J_1(n)[\mathfrak{m}]\to \mathbb{A}_{\mathbb{Q}}^1$ be a closed immersion of $\mathbb{Q}$-schemes constructed by a properly chosen non-constant rational function $\psi:X_1(n)\to \mathbb{P}_{\mathbb{Q}}^1$, then the heights of coefficients in the polynomial $P_{\iota}(X):=\prod_{x\in J_1(n)[\mathfrak{m}]\setminus O}(X-\iota(x))\in\mathbb{Q}[X]$ are bounded above by $n^{\delta}\cdot(\#\mathbb{F})^2$.

The main result of the paper is as follows.

\begin{theorem}\label{theorem:main}Let $n$ be a prime number. Given a plane model of $X_1(n)$. Assume the GRH, then $\#X_1(n)(\mathbb{F}_p)$ can be computed in time $\textrm{O}(n^{\delta+56+\epsilon}\cdot\log^{9+\epsilon}p)$.
\end{theorem}
\begin{remark}The constant $\omega$ does not occur, because the most time-consuming calculation turns out to be determining the matrix $\rho_f(\textrm{Frob}_p)$, rather than computing the polynomial $P_{\iota}(X)$. The Theorem is proved at the end of Section 2.
\end{remark}

The paper is organized as follows.
In Section 2, we describe the detailed strategy of reducing the main task of computing the characteristic polynomial $P_n(t)$ to the computation of the Galois representation associated to a single newform. More importantly, we explain how the GRH  promises a polynomial time algorithm. The main Theorem \ref{theorem:main} is proved there use the result in Algorithm \ref{complexityofsinglefactor} in Section 3.

In Section 3, we focus on computing the Galois representation associated to a newform. A more careful estimation of the generators of the maximal ideal is given there. We describe the main  algorithm there too.

In Section 4, we provide some real computations. All of our computations are based on Magma computational algebra system \cite{Bosma}.

\section{The strategy}
Let $S_2(\Gamma_1(n))$ be the space of cusp forms of weight two level $n$. Then we have a decomposition $S_2(\Gamma_1(n))=\bigoplus_{f}\textrm{span}(f)$
, where $f$ runs through a set of representatives for the Galois conjugacy classes of newforms in $S_2(\Gamma_1(n))$. Accordingly, the Jacobian $J_1(n)$ can be decomposed as $\prod_{f}A_f$, where $A_f$ is an abelian variety associated to $f$ with dimension $[\mathbb{Q}_f:\mathbb{Q}]$. The characteristic polynomial of $\textrm{Frob}_p$ on $J_1(n)$, denoted as $P_n(t)$ can be decomposed as $P_n(t)=\prod_{f}P_f(t)$, where $P_f(t)$ is the characteristic polynomial of $\textrm{Frob}_p$ on the abelian variety $A_f$.

We will recover $P_n(t)$ by computing each factor $P_f(t)$ one by one.

Let $\mathfrak{p}$ be one of the prime ideals of $\mathcal{O}_f$ over prime $\ell$ and $\mathbb{F}$ the residue field of $\mathfrak{p}$. We have a surjective  ring homomorphism $\theta_f:\mathbb{T}(n,2)\to \mathbb{F}$ and an associated  mod-$\ell$ Galois representation denoted as $\rho_f:\textrm{Gal}(\overline{\mathbb{Q}}/\mathbb{Q})\to \textrm{GL}_2(\mathbb{F})$. The representation can be realized by the $\mathbb{F}$-vector space $J_1(n)[\mathfrak{m}]$, here $\mathfrak{m}=\textrm{ker} \theta_f$ is a maximal ideal of the Hecke algebra $\mathbb{T}(n,2)$. In fact $J_1(n)[\mathfrak{m}]\subset J_1(n)[\ell]$ lands inside $A_f[\ell]$. In general, the dimension of $J_1(n)[\mathfrak{m}]$ as an $\mathbb{F}$-space is two, which is $2\cdot[\mathbb{F}:\mathbb{F}_\ell]$ as an $\mathbb{F}_\ell$-vector space. In the worst case, the residue degree $[\mathbb{F}:\mathbb{F}_\ell]$ might be close to the dimension of $J_1(n)$. Thus the cardinality of $\textrm{GL}_2(\mathbb{F})$ might be in $\textrm{O}(\ell^{4g})$. In this case, using the algorithm in \cite{Zeng}, the complexity of computing the mod-$\ell$ representation $\rho_f$ is already exponential in $n$. In order to get a polynomial time algorithm, we only consider those primes $\mathfrak{p}$ whose norms are bounded polynomial in $n$. The effective Chebotarev density theorem tells us that there exists sufficiently many such prime ideals of $\mathbb{Q}_f$, as follow.

\begin{theorem}(Weinberger)For $K$ a number field and $x$ a real number, let $\pi(x,K)$ denote the number of maximal ideals $\wp$ of the ring of integers $\mathcal{O}_K$ of $K$ with $\#(\mathcal{O}_K/\wp)\le x$. For $x>2$ in $\mathbb{R}$, let $\textrm{li } x=\int_2^x(1/\log y)dy$. Then there exists $c_1\in \mathbb{R}$ such that for every number field $K$ for which GRH holds, and for every $x>2$, one has
$$|\pi(x,K)-\textrm{li }x|\le c_1\cdot\sqrt{x}\cdot\log(|\textrm{discr}(\mathcal{O}_K)\cdot x^{\dim_{\mathbb{Q}}K}|).$$
\end{theorem}

The following corollary can be found in \cite{Edixhoven}(Corollary 15.2.8).

\begin{cor}\label{corollary:numberofgoodprimes}There exist absolute constants $c_2$ and $c_3$ in $\mathbb{R}$ such that for every number field $K$ for which GRH holds and for every $x$ in $\mathbb{R}$ such that
\begin{equation}\label{errorTerms}
x>\max\{c_2\cdot(\log |\textrm{discr}(\mathcal{O}_K)|)^2\cdot(\log(1+\log|\textrm{discr}(\mathcal{O}_K)|))^2, c_3\cdot (\dim_{\mathbb{Q}}K)^2\cdot(1+\log \dim_{\mathbb{Q}}K)^4 \},
\end{equation}
then
$$\pi_1(x,K)\ge \frac{1}{2}\cdot\frac{x}{\log x}.$$
\end{cor}

In order to use the results above, we need to get an upper bound for $|\textrm{discr}(\mathcal{O}_K)|$ in (\ref{errorTerms}). As already pointed out in \cite{Edixhoven}, we can modify the proof there a little bit, and give the following.

\begin{prop}Let $n$ be a prime number and $\mathbb{T}(n,2)_{\mathbb{R}}:=\mathbb{R}\otimes\mathbb{T}(n,2)$ . Then we have
$$\log \textrm{Vol}(\mathbb{T}(n,2)_{\mathbb{R}}/\mathbb{T}(n,2))=\frac{1}{2}\log |\textrm{discr}(\mathbb{T}(n,2))|\le n^2\log n.$$
\end{prop}

\begin{proof}Using the Sturm  bound \cite{Stu87}, the Hecke algebra $\mathbb{T}(n,2)$ is a rank $g$ free $\mathbb{Z}$-module generated by a set of Hecke operators $S:=\{T_{n_1},...,T_{n_g}\}$ with $n_i\le 2\cdot[\textrm{SL}_2(\mathbb{Z}):\Gamma_1(n)]/12=(n^2-1)/6$. Let $f_1,\ldots,f_g$ be the newforms of $S_2(\Gamma_1(n))$, then we have a ring isomorphism of $\mathbb{R}$-algebras
$$f:\mathbb{T}(n,2)_{\mathbb{R}}\to {\mathbb{R}}^g,t\mapsto(f_1(t),...,f_g(t)).$$
So
$$\textrm{Vol}(\mathbb{T}(n,2)_{\mathbb{R}}/\mathbb{T}(n,2))=|\det(f(T_{n_1}),...,f(T_{n_g}))|.$$
Using Deligne's bound, we have the $j$-th component of vector $f(T_{n_i})$ satisfying
$$|(f(T_{n_i}))_j|=|a_{n_i}(f_{j})|\le \sigma_0(n_i)n_i^{1/2}\le 2\cdot n_i.$$
Hence the square of the length of $f(T_{n_i})$ is at most $4gn_i^2$, thus
$$|\det(f(T_{n_1}),...,f(T_{n_g}))|\le\prod_{i=1}^g 2g^{1/2}n_i \le 2^g\cdot g^{g/2}\cdot(\frac{n^2-1}{6})!.$$
Hence
$$\log |\det(f(T_{n_1}),...,f(T_{n_g}))|\le g\log 2+\frac{g}{2}\log g+\log((\frac{n^2-1}{6})!).$$
Using $\log ((\frac{n^2-1}{6})!)\le (\frac{n^2-1}{6})\log(\frac{n^2-1}{6})-\frac{n^2-1}{6}+1$ and $g=\frac{(n-5)(n-7)}{24}$, the result follows.
\end{proof}

Since  $\mathcal{O}_f$ can be viewed as  a subring of $\mathbb{T}(n,2)$, we have $\log|\textrm{discr}(\mathcal{O}_f)|\le 2n^2\log n$. Moreover, since $\mathcal{O}_f\subset\mathbb{Z}_f$,  we have $\log |\textrm{discr}(\mathbb{Z}_f)|\le 2n^2\log n$. In other words, if
\begin{equation}
x>\max\{c_2\cdot (2n^2\log n)^2\cdot\log^2(1+2n^2\log n) ,c_3 \cdot n_f^2\cdot (1+\log n_f)^4  \},
\end{equation}
then $\pi(x,\mathbb{Q}_f)\ge\frac{x}{2\log x}$, where $n_f=[\mathbb{Q}_f:\mathbb{Q}]$. Moreover the extension degree $n_f$ is bounded above by  $\dim_{\mathbb{C}} S_2(\Gamma_1(n))=\frac{(n-5)(n-7)}{24}$. Thus, there exists an absolute constant $c_4$ such that, if $x>c_4\cdot n^4\log^4n$ then $\pi(x,\mathbb{Q}_f)\ge\frac{x}{2\log x}$.

Define $\pi(x,\mathcal{O}_f)=\{\mathfrak{p}:\textrm{prime ideal of } \mathcal{O}_f \textrm{ with norm } \le x \}.$ We can get a lower bound of $|\pi(x,\mathcal{O}_f)|$ as follow: Let $d=[\mathbb{Z}_f:\mathcal{O}_f]$, then for  any $\ell\nmid d$, there is a one-to-one correspondence between \{ prime ideals of $\mathbb{Z}_f$ over $\ell$ \} and \{ prime ideals of $\mathcal{O}_f$ over $\ell$\}. Thus we have $|\pi(x,\mathcal{O}_f)|\ge|\pi(x,\mathbb{Q}_f)|-B$, where $B$ is the number of prime ideals of $\mathbb{Q}_f$ over $d$. Since $\log|\textrm{discr}(\mathcal{O}_f)|\le 2n^2\log n$ and $|\textrm{discr}(\mathcal{O}_f)|=d^2\cdot|\textrm{discr}(\mathbb{Z}_f)|$, we have $\log d\le n^2\log n$. Thus there are at most $\log_2 d$ distinct primes dividing $d$. So we have  $B\le n_f\cdot\log_2 d$. So if $x\ge c_4n^4\log^4 n$, we have
$$|\pi(x,\mathcal{O}_f)|\ge|\pi(x,\mathbb{Q}_f)|-B\ge \frac{x}{2\log x}-n_f\cdot\log_2 d$$
Notice that $n_f\le\frac{(n-5)(n-7)}{24}$, we have $|\pi(x,\mathcal{O}_f)|\ge \frac{x}{3\log x}$, while the constant $c_4$ is large enough.

Let $p$ be a prime number not equal to $n$ or $\ell$, then the mod-$\ell$ representation $\rho_f:\textrm{Gal}(\overline{\mathbb{Q}}/\mathbb{Q})\to \textrm{GL}_2(\mathbb{F})$ is unramified at $p$. Let $R_\mathfrak{p}(t)$ be the characteristic polynomial of $\rho_f(\textrm{Frob}_p)$, then we have $R_\mathfrak{p}(t)\in\mathbb{F}[t]$. Let $A_\mathfrak{p}(t)=\prod_{\sigma\in\textrm{Gal}(\mathbb{F}/\mathbb{F}_\ell)} R_\mathfrak{p}(t)^{\sigma}$, where $\sigma$ acts coefficient wise on $R_\mathfrak{p}(t)$. Then $A_\mathfrak{p}(t)$ is degree $2\cdot[\mathbb{F}:\mathbb{F}_\ell]$ polynomial in $\mathbb{F}_\ell[t]$ and a factor of $P_f(t) \mod \ell$ as well. The following Lemma in \cite{Couveignes}(Lemma 30) shows that we can recover $P_f(t)$ from sufficiently many pairs $(\ell,A_\mathfrak{p}(t))$.

\begin{lemma}Let $d\ge 2$ be an integer. Let $I$ be a positive integer and for every $i$ from 1 to $I$ let $N_i\ge 2$ be an integer and $A_i(X)$ a monic polynomial with integer coefficients and degree $a_i$ where $1\le a_i\le d$. We ass the coefficients in $A_i(X)$ lie in the interval $[0,N_i)$.

We assume there exists an irreducible polynomial $P(X)$ with degree $d$ and integer coefficients and naive height $\le H$ such that $P(X)\mod N_i$ is a multiple of $A_i(X)\mod N_i$ for all $i$.

We assume the $N_i$ are pairwise coprime and
$$\prod_{1\le i\le I} N_i^{a_i}>((2d)!)^{d+1}\cdot H^{d(d+1)}\cdot 2^{\frac{d^2(d+1)}{4}}.$$

Then $P(X)$ is the unique polynomial fulfilling all these conditions and it can be computed from the $(N_i,A_i)$ by a deterministic Turing machine in time polynomial in $d$, $\log H$ and $I$, and the $\log N_i$.
\end{lemma}

As described in \cite{Couveignes}, let $\mathcal{P}_d$ be the additive group of integer coefficient polynomials with degree $\le d$, $\rho_i:\mathcal{P}_d\to \mathbb{Z}[X]/(A_i,N_i)$ be the reduction map modulo the ideal $(A_i,N_i)$. We have the product map (surjective)
$$\rho=\prod_{1\le i\le I}\rho_i:\mathcal{P}_d\to \prod_{1\le i\le I}\mathbb{Z}[X]/(A_i,N_i).$$
Its kernel is a lattice $\mathcal{R}$ with index $\Theta=\prod_{1\le i\le I}N_i^{a_i}$ in $\mathcal{P}_d=\mathbb{Z}^{d+1}$. The polynomial $P$ is the shortest vector in the lattice $\mathcal{R}$ for the $L^2$ norm. Applying the LLL algorithm (see \cite{Cohen}) to the lattice $\mathcal{R}$, we can reconstruct $P$ in time $\textrm{O}(d^6\log^3 B)$, where $B\le \sqrt{I}\cdot\prod_{1\le i\le I} N_i$ is the largest length of the basis of $\mathcal{R}$ under $L^2$ norm.

The characteristic polynomial $P_f(t)$ is a degree $2\cdot n_f$ monic polynomial in $\mathbb{Z}[t]$ with roots all have absolute values $p^{1/2}$. Without loss of generality, we assume $P_f(t)$ is irreducible. Let $H$ be the maximal absolute values of coefficients of $P_f(t)$, we have $H<p^{2n_f}$. Set $x=2\cdot(n^6\log p)\cdot\log(n^6\log p)$, we will show below that
\begin{equation}\label{Verify}
\prod_{\mathfrak{p}\in\pi(x,\mathcal{O}_f)} \ell_{\mathfrak{p}}^{f_\mathfrak{p}}>((2n_f)!)^{n_f+1}\cdot H^{n_f(n_f+1)}\cdot 2^{\frac{n_f^2(n_f+1)}{4}},
\end{equation}
where $\ell_\mathfrak{p}$ is the character of $\mathcal{O}_f/\mathfrak{p}$  and $f_\mathfrak{p}$ is the residue degree of $\mathfrak{p}$. For simplicity, during the following calculations, the implied absolute constant will always assumed to be large enough. Taking logarithm, the left hand side of (\ref{Verify}) is
\begin{equation}
\begin{split}
L:=\sum_{\mathfrak{p}\in\pi(x,\mathcal{O}_f)}f_{\mathfrak{p}}\log \ell_\mathfrak{p} &\ge \sum_{\mathfrak{p}\in\pi(x,\mathcal{O}_f)}\log \ell_\mathfrak{p}\\
&\ge\frac{x}{3\log x}\log 2\ge \frac{\log 2}{3}\cdot n^6\cdot \log p.
\end{split}
\end{equation}
Here we use the fact that if  $a\in\mathbb{R}_{>e}$ and $x>2a\log a$, then $x/\log x>a$.  The right hand side of (\ref{Verify}) is
\begin{equation}
\begin{split}
R:=&(n_f+1)\log((2n_f)!)+n_f(n_f+1)\log H+\frac{n_f^2(n_f+1)}{4}\log 2\\
&<(n_f+1)2n_f\log(2n_f)+2n_f^2(n_f+1)\log p+\frac{n_f^2(n_f+1)}{4}\log 2.
\end{split}
\end{equation}
Since $n_f\le\frac{(n-5)(n-7)}{24}$, we can see $R<L$.

For each $\mathfrak{p}\in\pi(x,\mathcal{O}_f)$, using the Algorithm \ref{complexityofsinglefactor} in Section 3, the characteristic polynomial $A_\mathfrak{p}(t)=\prod_{\sigma\in\textrm{Gal}(\mathbb{F}/\mathbb{F}_\ell)} R_\mathfrak{p}(t)^{\sigma}$ can be computed in time
$$C_\mathfrak{p}:=\textrm{O}(n^{2+\delta+2\omega+\epsilon}\cdot |\mathbb{F}|^{3+\epsilon}\cdot(n^2+|\mathbb{F}|)+n^{\delta+\epsilon}\cdot|\mathbb{F}|^{8+\epsilon}
+|\mathbb{F}|^5\cdot\log^{2+\epsilon}p).$$

Let $x=2\cdot(n^6\log p)\cdot\log(n^6\log p)$, the time of computing all $A_\mathfrak{p}(t)$, $\mathfrak{p}\in\pi(x,\mathcal{O}_f)$ is
\begin{displaymath}
\begin{split}
\sum_{\mathfrak{p}\in\pi(x,\mathcal{O}_f)} C_\mathfrak{p}&\le \sum_{\mathfrak{p}\in\pi(x,\mathcal{O}_f)} \textrm{O}(n^{2+\delta+2\omega+\epsilon}\cdot x^{3+\epsilon}\cdot(n^2+x)+n^{\delta+\epsilon}\cdot x^{8+\epsilon}
+x^5\cdot\log^{2+\epsilon}p)\\
&\le\sum_{\mathfrak{p}\in\pi(x,\mathcal{O}_f)}\textrm{O}(n^{48+\delta+\epsilon}\cdot \log^{8+\epsilon}p) \le \textrm{O}(n^{\delta+54+\epsilon}\cdot\log^{9+\epsilon}p).
\end{split}
\end{displaymath}
So, $P_f(t)$ can be computed in time $\textrm{O}(n^{\delta+54+\epsilon}\cdot\log^{9+\epsilon}p)$. Since $P_n(t)=\prod_{f}P_f(t)$, and there are at most $g=\frac{(n-5)(n-7)}{24}$ distinct newforms in $S_2(\Gamma_1(n))$. Thus the time of computing $P_n(t)$ is at most
$$g\cdot\textrm{O}(n^{\delta+54+\epsilon}\cdot\log^{9+\epsilon}p)=\textrm{O}(n^{\delta+56+\epsilon}\cdot\log^{9+\epsilon}p).$$

So we prove the Theorem \ref{theorem:main}.

The following sections are devoted to the computation of a single Galois representation $\rho_{f}:\textrm{Gal}(\overline{\mathbb{Q}}/\mathbb{Q})\to \textrm{GL}_2(\mathbb{F})$.

\section{Computing the mod-$\ell$ representation}

Let $f\in S_2(\Gamma_1(n))$ be a newform. In this section, we will concentrate on the computation of a single mod-$\ell$ Galois representation, given by a homomorphism $\theta_f:\mathbb{T}(n,2)\to \mathcal{O}_f/\mathfrak{p}\simeq\mathbb{F}$. Our goal is to determine the $\mathbb{F}$-vector space $J_1(n)[\ker \theta_f]$ explicitly. Using the Sturm bound, the maximal ideal $\mathfrak{m}:=\ker \theta_f$ can be generated by $\ell$ and $T_k-\theta_f(T_k)$ with $1\le k\le \frac{n^2-1}{6}$ as an $\mathbb{Z}$-module. And $J_1(n)[\mathfrak{m}]$ is exactly
$$J_1(n)[\mathfrak{m}]=\bigcap_{1\le k\le\frac{n^2-1}{6}}\ker(T_k-a_k(f),J_1(n)[\ell]).$$
So a straightforward idea to track $J_1(n)[\mathfrak{m}]$ might be the following. First, we compute a base of $J_1(n)[\ell]$ an then represent each action of $T_k-a_k(f)$ by a $1\times 2g$ column. All the actions $T_k-a_k(f),1\le k \le \frac{n^2-1}{6}$ are represented by a matrix of size nearly $\frac{n^2-1}{6}\times 2g$. Then we have $J_1(n)[\mathfrak{m}]$ by solving the equations. The size of the matrix we have to manage is close to $n^2\times n^2$, which is very big. However, there is another way to do the computation, proposed by Couveignes in \cite{Couveignes}. We describe this method below.

Let us begin with a carefully estimating of the generators of $\mathfrak{m}$. As we will see, the estimation essentially boils down to the question: how many terms will be sufficed to distinguish two newforms with same level and weight. And we have the following fact, as a special case of Corollary 1.7 in \cite{Buzzard}.

\begin{lemma}For $n$ and $\ell$ two distinct primes. Let $f_1$ and $f_2$ be two newforms of $S_2(\Gamma_1(n),\overline{\mathbb{F}}_\ell)$, satisfying that, the first $n$  terms of the $q$-expansions of $f_1$ and $f_2$ agree. Then $f_1=f_2$.
\end{lemma}
\begin{proof}Using the Corollary 1.7 in \cite{Buzzard}. It suffices to prove the character of $f_1$ and $f_2$ are the same. For prime $p$, we have
$$a_{p^2}(f_1)=a_p(f_1)^2-p\cdot\chi_1(p),a_{p^2}(f_2)=a_p(f_2)^2-p\cdot\chi_2(p),$$
where $\chi_1$ and $\chi_2$  are the characters associated to $f_1$ and $f_2$ respectively. Hence, for any $p^2<n$, we have $\chi_1(p)=\chi_2(p)$. On the other hand, there exists a prime number generating of the group $(\mathbb{Z}/n\mathbb{Z})^\times$ and satisfying $p^2<n$ as well. Let $q$ be such a prime, then we have $\chi_1(q)=\chi_2(q)$. Which implies $\chi_1=\chi_2$.
\end{proof}

A direct corollary of the lemma above is the following fact.

\begin{cor}Let $n$ be a prime number, $f\in S_2(\Gamma_1(n))$  a newform and $\theta_f:\mathbb{T}(n,2)\to\mathbb{F}$ a surjective ring homomorphism associated to $f$, where $\char(\mathbb{F})\not=n$. Then the maximal ideal $\mathfrak{m}:=\ker \theta_f$ can be generated by $\ell$ and $T_k-a_k(f)$ with $1\le k\le n$ as an $\mathbb{T}(n,2)$-module.
\end{cor}

\begin{proof}We have isomorphism
$$\mathbb{T}(n,2)\otimes\mathbb{F}_\ell\cong\prod_{g}\mathbb{F}_\ell(g),$$
where $g$ runs through a set of representatives for the $\Gal(\overline{\mathbb{F}}_\ell/\mathbb{F}_\ell)$-conjugate classes of newforms in $S_2(\Gamma_1(n),\overline{\mathbb{F}}_\ell)$ and $\mathbb{F}_\ell(g)$ is the field generated by coefficients of $g$.

Now to prove the Corollary, it is enough to show for every representative newform $g$ and integer $m>n$, there exist $c_k\in \mathbb{F}_\ell(g)$ such that
$$a_m(g)-a_m(f)=\sum_{k=1}^n c_k\cdot(a_k(g)-a_k(f)).$$
This is true by the Lemma above.
\end{proof}

Moreover, a subset of $\{\ell, T_k-a_k(f), 1\le k\le n\}$ may also suffice to generate $\mathfrak{m}$, so we define.
\begin{definition}\label{definition:optimalset}Let $\mathcal{S}$ be a set of positive integers, such that $\ell$ and $T_k-a_k(f),k\in\mathcal{S}$ generate $\mathfrak{m}$ as an $\mathbb{T}(n,2)$-module. The set $\mathcal{S}$ is called optimal if $\prod_{k\in \mathcal{S}}k$ is minimal among all the $\mathcal{S}$.
\end{definition}
\begin{remark}In practice, the optimal set $\mathcal{S}$  contains only several small integers, much smaller compared to the theoretical bound. For example, when $n=13$, $\ell=7$, $\mathcal{S}=\{2\}$, $n=13$, $\ell=5$, $\mathcal{S}=\{2\}$ and $n=17$, $\ell=17$, $\mathcal{S}=\{2\}$.
\end{remark}

For every integer $k\ge 2$, the characteristic polynomial of $T_k$ acting on $S_2(\Gamma_1(n))$ is a degree $g$ monic polynomial belonging to $\mathbb{Z}[X]$. We denote it by $A_k(X)$, which can be factored as
$$A_k(X)\equiv B_k(X)(X-a_k(f))^{e_k}~\mod \mathfrak{p},$$
with $B_k(X)$ monic, $B_k(a_k(f))\not=0 \in \mathbb{F}$ and $e_k\ge 1$. We call $\pi_k:J_1(n)[\ell]\to J_1(n)[\ell],P\mapsto B_k(T_k)(P)$ the projection map. Which maps bijectively $J_1(n)[\mathfrak{m}]$ onto itself. Define the composition map as $\pi_\mathcal{S}:=\prod_{k\in\mathcal{S}}\pi_k$. Then for each point $P\in J_1(n)[\ell]$, a properly adjustment of  $\pi_\mathcal{S}(P)$ (applying the action $T_k-a_k(f)$ for $k\in \mathcal{S}$ when needed) gives a point in $J_1(n)[\mathfrak{m}]$, denoted as $\tilde{\pi}_\mathcal{S}(P)$.

The explicitly calculation is realized by computing sufficiently many $J_1(n)[\mathfrak{m}] \mod p$ first, and then recover $J_1(n)[\mathfrak{m}]$ by the Chinese Remainder Theorem.

\begin{theorem}Let $n$ and $\ell$ be a distinct prime numbers. Let $\mathbb{F}$ be a finite extension of $\mathbb{F}_\ell$. Assume  $\theta_f:\mathbb{T}(n,2)\to \mathbb{F}$ is a surjective ring homomorphism with kernel denoted as $\mathfrak{m}$. Assume $p$ is a good  prime not equal to $n$ or $\ell$.
Given a plane model of the modular curve  $X_1(n)$ and the Zeta function of $X_1(n)_{\mathbb{F}_p}$. Then a base of the two dimensional $\mathbb{F}$-vector space $J_1(n)[\mathfrak{m}]\mod p$ can be computed in time
$$\textrm{O}(n^{2+2\omega+\epsilon}\log^{2+\epsilon}q)+\textrm{O}(n^{4+2\omega+\epsilon}\log^{1+\epsilon}q),$$
where $q=p^{|\mathbb{F}^\times|}$.
\end{theorem}
\begin{remark}As mentioned in Section 1, $\omega$ is a constant in $[2,4]$ and $\epsilon$ is any positive real number. A prime $p$ is called good prime if it simultaneously satisfies: (1) The order of $X$ in $\mathbb{F}[X]/(F(X))$ is at most $|\mathbb{F}^\times|$, where $F(X):=X^2-a_p(f)X+\chi_f(p)\cdot p$ is the characteristic polynomial of $\textrm{Frob}_p$ on $J_1(n)[\mathfrak{m}]\mod p$. (2) Prime $p$ is $\mathfrak{m}$-good, see \cite{Bruin} for the detail.

This is one of the main results in \cite{Zeng}. The only different is that, here $J_1(n)[\mathfrak{m}]\mod p$  is a two dimensional $\mathbb{F}$-vector space, while in \cite{Zeng} $J_1(n)[\mathfrak{m}]\mod p$ is a two dimensional $\mathbb{F}_\ell$-vector space. The argument there can be easily extended to our cases .

\end{remark}

The calculation of $J_1(n)[\mathfrak{m}]\textrm{ mod } p$ can be split into the following steps.

\begin{algo}Computing $J_1(n)[\mathfrak{m}]\mod p$.

1. Compute the definition field of $J_1(n)[\mathfrak{m}]\mod p$ using the characteristic polynomial $X^2-a_p(f)X+\chi_f(p)\cdot p$ of $\textrm{Frob}_p$ on $J_1(n)[\mathfrak{m}]\mod p$, denoted by $\mathbb{F}_q$.

2. Compute $N_n:=\#J_1(n)(\mathbb{F}_q)$. Denote the prime-to-$\ell$ part of $N_n$ by $N_\ell$.

3. Construct an $\ell$-torsion point. We pick a random point $Q$ of $J_1(n)(\mathbb{F}_q)$ and compute $N_\ell\cdot Q$ (adjust it by multiplying $\ell$ several times when needed), denote the output $\ell$-torsion point by $P$.

4. Construct a point of $J_1(n)[\mathfrak{m}]\mod p$ by computing the projection $\tilde{\pi}_\mathcal{S}(P)$.

5. Construct a base of  $J_1(n)[\mathfrak{m}]\mod p$  by applying the Step 4 several times and checking linearly independence.
\end{algo}

Analogous to \cite{Zeng}, the complexity analysis here can be done without difficulty. We only remark here that the character $\chi_f$ associated to the newform $f$ can be computed by the formula
$$\chi_f(q)=\frac{a_q(f)^2-a_{q^2}(f)}{q} ~~,~~ \chi_f(m)=\chi_f(m~\textrm{mod}~n)~~\textrm{and}~~\chi_f(a\cdot b)=\chi_f(a)\cdot\chi_f(b),$$
where $q$ is any  prime number not equal to $\ell$.

Now, given sufficiently many $J_1(n)[\mathfrak{m}]\mod p$ for small primes $p$, we can recover $J_1(n)[\mathfrak{m}]$ in the following sense.

Let $O$ be a rational point of $X_1(n)$, which will be served as the origin of the Jacobi map. Any point $x\in J_1(n)[\mathfrak{m}]$ can be
represented uniquely by $D-\theta(x)\cdot O$, where $\theta(x)$ is the stability of $x$ (see \cite{Couveignes} for the definition) and $D$ is
an effective divisor of degree $\theta(x)$. Let $\psi(x)\in \mathbb{Q}(X_1(n))$ be a rational function, which has no pole except at $O$.  Define a
function $\iota:J_1(n)[\mathfrak{m}]\to\overline{\mathbb{Q}}$ as $\iota(x)=\psi(Q_1)+\ldots+\psi(Q_{\theta(x)})$, where $x$ is represented as
$x=Q_1+\ldots+Q_{\theta(x)}-\theta(x)\cdot O$. Then we have a polynomial describing $J_1(n)[\mathfrak{m}]$ as follow
$$P_\iota(X):=\sum_{x\in J_1(n)[\mathfrak{m}]\setminus O}(X-\iota(x)).$$
Which has degree $(\#\mathbb{F})^2-1$ and belongs to $\mathbb{Q}[X]$.

The heights of coefficients in $P_\iota(X)$ are well bounded. The following theorem can be found in \cite{Bruin} and \cite{Edixhoven}.

\begin{theorem}Notations as above. There exists a nonconstant function $\iota:J_1(n)[\mathfrak{m}]\to \overline{\mathbb{Q}}$ such that the heights of coefficients in $P_\iota(X)$ are bounded above by $\textrm{O}(n^\delta\cdot |\mathbb{F}|^2)$, where $\delta$ a constant independent of $n$ and $\mathfrak{m}$.
\end{theorem}

The function $\iota$ is constructed explicitly in \cite{Edixhoven} also in \cite{Bruin}. In practice, a carefully chosen of the function $\iota$ will lower the height, see \cite{Derickx} and \cite{Zeng} for the detail.

The reduction polynomial $P_\iota(X) \mod p$ can be computed from $J_1(n)[\mathfrak{m}]\mod p$ directly. So we can reconstruct $P_\iota(X)$ by the Chinese Remainder Theorem.
Since the heights of coefficients of $P_\iota(X)$ are expected to be in $\textrm{O}(n^{\delta}|\mathbb{F}|^2)$, it suffices to recover $P_\iota(X)$ from $P_\iota(X)\mod p$ 's with primes $p$ bounded above by $\textrm{O}(n^{\delta}\cdot|\mathbb{F}|^2)$. So we have the following.

\begin{theorem}Notations as above, $P_\iota(X)$ can be computed in time
$$\textrm{O}(n^{2+\delta+2\omega+\epsilon}\cdot |\mathbb{F}|^{3+\epsilon}\cdot(n^2+|\mathbb{F}|)).$$
\end{theorem}
\begin{proof}
Because
\begin{displaymath}
\begin{split}
&\sum_{p\le n^\delta\cdot |\mathbb{F}|^2}\textrm{O}(n^{2+2\omega+\epsilon}\cdot\log^{2+\epsilon} q)+\textrm{O}(n^{4+2\omega+\epsilon}\cdot\log^{1+\epsilon }q)\\
&=\sum_{p\le n^\delta\cdot |\mathbb{F}|^2} \textrm{O}(n^{2+2\omega+\epsilon}\cdot (|\mathbb{F}|\cdot\log p)^{2+\epsilon})+\textrm{O}(n^{4+2\omega+\epsilon}\cdot (|\mathbb{F}|\cdot\log p)^{1+\epsilon})\\
&\le \textrm{O}(n^{2+\delta+2\omega+\epsilon}\cdot |\mathbb{F}|^{4+\epsilon})
+\textrm{O}(n^{4+\delta+2\omega+\epsilon}\cdot |\mathbb{F}|^{3+\epsilon})\\
&=\textrm{O}(n^{2+\delta+2\omega+\epsilon}\cdot |\mathbb{F}|^{3+\epsilon}\cdot(n^2+|\mathbb{F}|)).
\end{split}
\end{displaymath}
The theorem follows.
\end{proof}

Once we have the polynomial $P_\iota(X)$, the Galois representation
$$\rho_f:\textrm{Gal}(\overline{\mathbb{Q}}/\mathbb{Q})\to \textrm{GL}_2(\mathbb{F}),$$
is clear. More precisely as in \cite{Zeng}, we first lift the roots $a_i,1\le i\le|\mathbb{F}|^2-1$ of $P_\iota(X)$ to sufficiently high precision by Hensel lifting method and then compute all the polynomials $\Gamma_C(X)$ corresponding to the conjugacy classes $C$ of $\textrm{GL}_2(\mathbb{F})$. The polynomial $\Gamma_C(X)$ is defined as
$$\Gamma_C(X)=\prod_{\sigma\in C}(X-\sum_{i=1}^{|\mathbb{F}|^2-1}h(a_i)\sigma(a_i) ),$$
where $h(X)\in\mathbb{Q}[X]$ is an auxiliary polynomial with small degree. For each conjugacy class $C\subset \textrm{GL}_2(\mathbb{F})$, we have
$$\rho_f(\textrm{Frob}_p)\in C\Leftrightarrow \Gamma_C(\textrm{Tr}_{\frac{\mathbb{F}_p[x]}{P(x)}/\mathbb{F}_p}(h(x)x^p))\equiv0\mod p.$$
Using the criterion above, we can determine the matrix $\rho_f(\textrm{Frob}_p)$.

\begin{algo}\label{complexityofsinglefactor} Let $p$ and $\ell$ be prime numbers and
$$\rho_f:\textrm{Gal}(\overline{\mathbb{Q}}/\mathbb{Q})\to \textrm{GL}_2(\mathbb{F})$$
the mod-$\ell$ Galois representation as above. Then the characteristic polynomial of $\rho_f(\textrm{Frob}_p)$ can be computed in time

$$\textrm{O}(n^{2+\delta+2\omega+\epsilon}\cdot |\mathbb{F}|^{3+\epsilon}\cdot(n^2+|\mathbb{F}|)+n^{\delta+\epsilon}\cdot|\mathbb{F}|^{8+\epsilon}
+|\mathbb{F}|^5\cdot\log^{2+\epsilon}p),$$
by the following steps.

1. Compute the polynomial $P_\iota(X)=\sum_{x\in J_1(n)[\mathfrak{m}]\setminus O}(X-\iota(x))$.

2. Lift all the roots $a_i$ of $P_\iota(X)$ to a precision of $n^\delta\cdot|\mathbb{F}|^4$.

3. Compute the polynomials $\Gamma_C(X)$ for all the conjugacy classes of $\GL_2(\mathbb{F})$.

4. Determine the matrix $\rho_f(\textrm{Frob}_p)$ by checking whether $$\Gamma_C(\textrm{Tr}_{\frac{\mathbb{F}_p[x]}{P(x)}/\mathbb{F}_p}(h(x)x^p))\equiv0\mod p.$$
\end{algo}

Using Hensel lifting lemma, the complexity of lifting a root of a degree $d$ polynomial to precision $k$ is $\textrm{O}(d\cdot k^{1+\epsilon})$.
So the complexity of Step 2 is $|\mathbb{F}|^2\cdot\textrm{O}(|\mathbb{F}|^2\cdot(n^\delta|\mathbb{F}|^4)^{1+\epsilon})=\textrm{O}((n^\delta|\mathbb{F}|^8)^{1+\epsilon})$.
Given the lifted roots $a_i$, each polynomial $\Gamma_C(X)$ can be computed in time $\textrm{O}((|\mathbb{F}|^2\cdot( n^\delta|\mathbb{F}|^4))^{1+\epsilon})=\textrm{O}(( n^\delta|\mathbb{F}|^6)^{1+\epsilon})$. There are $|\mathbb{F}|^2-1$ conjugacy classes of $\GL_2(\mathbb{F})$ with length in $\textrm{O}(|\mathbb{F}|^2)$. So all the polynomials can be computed in $\textrm{O}(( n^\delta|\mathbb{F}|^8)^{1+\epsilon})$. Step 4 can be finished in $\textrm{O}(|\mathbb{F}|^5\cdot\log^{2+\epsilon}p)$ as showed in \cite{Zeng}.

Hence the total complexity is
$$\textrm{O}(n^{2+\delta+2\omega+\epsilon}\cdot |\mathbb{F}|^{3+\epsilon}\cdot(n^2+|\mathbb{F}|))+\textrm{O}(n^{\delta+\epsilon}\cdot|\mathbb{F}|^{8+\epsilon}
+|\mathbb{F}|^5\cdot\log^{2+\epsilon}p).$$

Now we can state the main algorithm.

\begin{algo}Compute the characteristic polynomial $P_n(t)$ of $\textrm{Frob}_p$ on $J_1(n)$.

1. Initialization. Let $f_1,\ldots,f_r$ be a set of representatives of $\Gal(\overline{\mathbb{Q}}/\mathbb{Q})$-conjugate classes of newforms in $S_2(\Gamma_1(n))$, with number field $\mathbb{Q}_{f_i}$ and extension degree $d_i:=[\mathbb{Q}_{f_i}:\mathbb{Q}]$ respectively. Set $x:=2n^6\log p\cdot\log(n^6\log p)$.

2. For each newform $f_i$, compute the set $\pi(x,\mathcal{O}_{f_i})$ by factoring ideals $\ell\mathcal{O}_{f_i}$, where $\ell$ is prime and $ 2\le \ell\le x$. Set $\mathcal{M}=\emptyset$.

3. For each $\mathfrak{p}\in \pi(x,\mathcal{O}_{f_i})$, denote $\mathbb{F}={\mathcal{O}_{f_i}}/\mathfrak{p}$ and $\ell=\char(\mathbb{F})$, compute the representation
$$\rho_{f_i}:\textrm{Gal}(\overline{\mathbb{Q}}/\mathbb{Q})\to \textrm{GL}_2(\mathbb{F}),$$
by Algorithm \ref{complexityofsinglefactor}.
Compute the characteristic polynomial of $\rho_{f_i}(\textrm{Frob}_p)$, denoted as  $R_\mathfrak{p}(t)$.
Compute $A_\mathfrak{p}(t)=\prod_{\sigma\in\textrm{Gal}(\mathbb{F}/\mathbb{F}_\ell)}R_\mathfrak{p}(t)^{\sigma}$ and update $\mathcal{M}=\mathcal{M}\cup\{(\ell,A_\mathfrak{p}(t))\}$.

3. Reconstruct the characteristic polynomial $P_{f_i}(t)$ of $\textrm{Frob}_p$ on the abelian variety $A_{f_i}$ from the data $\mathcal{M}$.

4. Output $P_n(t):=\prod_{i=1}^rP_{f_i}(t)$.

\end{algo}

\begin{remark}
As we have seen, the computation of all the polynomials $\Gamma_C(X)$ takes not only a lot of time but also a lot of memory.  Because for each representation $\rho_f:\textrm{Gal}(\overline{\mathbb{Q}}/\mathbb{Q})\to \textrm{GL}_2(\mathbb{F})$ , there are nearly $|\mathbb{F}|^2$ polynomials needed to recover. Moreover, the required precision is in $\textrm{O}(n^\delta|\mathbb{F}|^4)$, which is rather large. It is better to reduce the original polynomial $P_\iota(X)$ to a polynomial with small coefficients first. But a new problem arises: it is hard to reduce a polynomial, which has both huge coefficients and large degree ($=|\mathbb{F}|^2-1$).

To be rigorous we should get upper bounds of the time needed in Step 2 and Step 4.

The factorization of  ideal $\ell\mathcal{O}_{f_i}$ can be determined via polynomial factorization over finite field $\mathbb{F}_\ell$ with a complexity $\textrm{O}((d_i^{2+\epsilon}+d_i\log^{1+\epsilon} \ell)\cdot\log^{1+\epsilon}\ell )$ bit operations, see \cite{Gathen}. So the complexity of computing the set $\pi(x,\mathcal{O}_{f_i}) $ is
$$\sum_{\mathfrak{p}\in\pi(x,\mathcal{O}_{f_i})} \textrm{O}((d_i^{2+\epsilon}+d_i\log^{1+\epsilon} \ell)\cdot\log^{1+\epsilon}\ell ).$$
Which is bounded above by $\textrm{O}(n^{12+\epsilon}\log^{1+\epsilon}p)$. Notice that $d_i<g=\frac{(n-5)(n-7)}{24}$.

As showed in Section 2, We have that the running time of Step 3 is bounded above by $\textrm{O}(d_i^6\log ^3B)$, where $$B=\sqrt{|\mathcal{M}|}\cdot\prod_{(\ell,A_\mathfrak{p}(t))\in\mathcal{M}}\ell.$$
Thus
\begin{displaymath}
\log B=\frac{1}{2}\log|\mathcal{M}|+\sum_{(\ell,A_\mathfrak{p}(t))\in\mathcal{M}}\log\ell\le\frac{1}{2}\log x+x\log x.
\end{displaymath}
So the complexity of reconstructing $P_{f_i}(t)$ from the data $\mathcal{M}$ is bounded above by $\textrm{O}(n^{30+\epsilon}\log^{3+\epsilon} p)$.

\end{remark}

\section{Implementation and results}
The algorithm has been implemented in MAGMA. The following computation was done on a personal computer AMD FX(tm)-6200 Six-Core Processor 3.8GHz. We illustrate our algorithm by computing the characteristic polynomials (mod-$\ell$) of $\textrm{Frob}_p$ on $J_1(13)$ and $J_1(17)$, where $p$ is equal to $10^{1000}+1357$.

\begin{example}
Level $n=13$, $\ell=7$ and $p=10^{1000}+1357$.

The two dimensional space $S_2(\Gamma_1(13))$ is span by a single newform, denoted as
$$f(q)=q + (-\alpha - 1)\cdot q^2 + (2\alpha - 2)\cdot q^3 + \alpha\cdot q^4 + (-2\alpha + 1)\cdot q^5 + (-2\alpha + 4)\cdot q^6 + O(q^8),$$
where $\alpha$ is a root of $x^2 - x + 1$. The field generated by coefficients of $f$ is $\mathbb{Q}_f=\mathbb{Q}(\alpha)$. We have $\mathbb{Z}_f=\mathcal{O}_f=\mathbb{Z}[\alpha]$.
The prime $\ell=7$ splits completely in $\mathbb{Q}_f$, and the reductions of $f$ modulo $7$ contain two newforms with coefficients in $\mathbb{F}_{7}$, as follow
\begin{displaymath}
\begin{split}
&f_1(q)=q + q^2 + q^3 + 5q^4 + 5q^5 + q^6 + O(q^8),\\
&f_2(q)=q + 3q^2 + 4q^3 + 3q^4 + 2q^5 + 5q^6 + O(q^8).
\end{split}
\end{displaymath}
Denote the surjective morphism and mod-7 Galois representation associated to $f_i$ as $\theta_{f_i}:\mathbb{T}(n,2)\to\mathbb{F}_\ell$ and $\rho_{f_i}:\Gal(\overline{\mathbb{Q}}/\mathbb{Q})\to \GL_2(\mathbb{F}_{7})$ respectively. Using the Algorithm \ref{complexityofsinglefactor}, we computed the polynomial $P_i(X):=\prod_{x\in J_1(13)[\mathfrak{m}_i]\setminus O}(X-\iota(x))$. Let $q_i$ be the common denominator of coefficients of $P_i(X)$, then we have $q_1=11,q_2=23$ and
\begin{displaymath}
\begin{split}
11\cdot P_1(X)=&11 X^{ 48 }-1536 X^{ 47 }-99963 X^{ 46 }-2770225 X^{ 45 }-47983708 X^{ 44 }-591629167 X^{ 43 }\\
         &-5530796666 X^{ 42 }-40518974765 X^{ 41 }-235718747246 X^{ 40 }-1079974808307 X^{ 39 }-\\
         &3703420410826 X^{ 38 }-7605759696314 X^{ 37 }+7722875347212 X^{ 36 }+163461885735705 X^{ 35 }\\
         &+933845715340150 X^{ 34 }+3729957980566605 X^{ 33 }+11830678171677348 X^{ 32 }+\\
         &31022907406623640 X^{ 31 }+68071280480939026 X^{ 30 }+124294148205665813 X^{ 29 }+\\
         &183540094118403894 X^{ 28 }+200761571113722791 X^{ 27 }+106064052441293580 X^{ 26 }-\\
         &161268382392065434 X^{ 25 }-594425922266374264 X^{ 24 }-1074949608174189506 X^{ 23 }-\\
         &1394958011053869132 X^{ 22 }-1361031410365411144 X^{ 21 }-922433052155682350 X^{ 20 }-\\
         &224749370197814451 X^{ 19 }+464791623500074770 X^{ 18 }+910668196246156905 X^{ 17 }+\\
         &1027292221653623062 X^{ 16 }+885961994116863787 X^{ 15 }+631053161001956375 X^{ 14 }+\\
         &385831413164070220 X^{ 13 }+208737752200416506 X^{ 12 }+103130709785233526 X^{ 11 }+\\
         &48018067928428294 X^{ 10 }+21474802798164111 X^{ 9 }+9156461475778523 X^{ 8 }+\\
         &3598739833720301 X^{ 7 }+1245388931877955 X^{ 6 }+362616848767471 X^{ 5 }+\\
         &85245647038024 X^{ 4 }+15591113808897 X^{ 3 }+2146924473929 X^{ 2 }+\\
         &212924174747 X+12610968353.
\end{split}
\end{displaymath}

The matrices (up to similarity transformation) $\rho_{f_1}(\textrm{Frob}_p)$ and $\rho_{f_2}(\textrm{Frob}_p)$ are $\left[\begin{smallmatrix} 2 & 5 \\ 2 & 0 \end{smallmatrix}\right]$ and $\left[\begin{smallmatrix} 5 & 2 \\ 5 & 0 \end{smallmatrix}\right]$ respectively. Hence, the characteristic polynomial of $\textrm{Frob}_p$ on $J_1(13)$ satisfying,
$$P(t)\equiv(t^2-2t-10)(t^2-5t-10)\equiv t^4 + 2t^2 + 4\mod 7$$
so we have
$$\#X_1(13)(\mathbb{F}_p)\mod 7=(1+0+p)\mod 7=4$$
and
$$\#J_1(13)(\mathbb{F}_p)\mod 7=P(1)=0$$
\end{example}

\begin{example}
Level $n=17$, $\ell=17$ and $p=10^{1000}+1357$.

There are five newforms in
 $S_2(\Gamma_1(17))$ consisting of two $\Gal(\overline{\mathbb{Q}}/\mathbb{Q})$-conjugate
 classes. A set of representatives are
\begin{displaymath}
\begin{split}
&f_1(q)=q - q^2 - q^4 - 2q^5 + 4q^7 + O(q^8),\\
&f_2(q)=q + (-\alpha^3 + \alpha^2 - 1)\cdot q^2 + (\alpha^3 - \alpha^2 -\alpha - 1)\cdot q^3 + (2\alpha^3 - \alpha^2 + 2\alpha)\cdot q^4
    + O(q^5),
\end{split}
\end{displaymath}
where $\alpha$ is a root of $x^4 + 1$. The field generated by coefficients of $f_2$ is $\mathbb{Q}_{f_2}:=\mathbb{Q}(\alpha)$. We have $\mathbb{Z}_{f_2}=\mathcal{O}_{f_2}=\mathbb{Z}[\alpha]$. In fact $f_1$ is the newform in $S_2(\Gamma_0(17))$, corresponding to an elliptic curve over $\mathbb{Q}$, with a Weiestrass equation
$$E:y^2 + xy + y = x^3 - x^2 - x - 14.$$
The characteristic polynomial (mod $\ell$) of $\textrm{Frob}_p$ on $E$ is
$$P_{f_{\ell,1}}(t)=t^2-a_p\cdot t+p\mod \ell$$
where $a_p=p+1-\#E(\mathbb{F}_p)$. We can use elliptic point counting to get the number $\#E(\mathbb{F}_p)$. Which has a complexity $\textrm{O}(\log^{4+\epsilon}p)$, see \cite{Schoof}. We have $P_{f_\ell,1}(t)=t^2-11t+13$.

On the other hand, similar to the method described in Section 3, we can determine $P_{f_{\ell,1}}(t)$ by computing the Galois representation associated to the elliptic curve $E$. Namely, we compute
$$\rho_{E,\ell}:\Gal(\overline{\mathbb{Q}}/\mathbb{Q})\to \textrm{Aut}(E[\ell])\cong\GL_2(\mathbb{F}_{\ell}).$$
An interesting observation is that, we can compute $\#E(\mathbb{F}_p)$ in time polynomial in $\log p$ by computing sufficiently many mod-$\ell$ Galois representation associated to $E$.

The prime $\ell=17$ splits completely in the number field $\mathbb{Q}_{f_2}$. And the reductions of $f_2(q)$ modulo $\ell$ contain 4 newforms as follow
\begin{displaymath}
\begin{split}
&f_{2,1}(q)=q + 11q^2 +  6q^3 + 10q^4 + 4q^5 + 15q^6 + 9 q^7 + 11q^8 + 8 q^9 + 10q^{10}+ 2 q^{11} + \textrm{O}(q^{12}),\\
&f_{2,2}(q)=q + 14q^2 +  9q^3 +   q^4 + 6q^5 + 7 q^6 + 6 q^7 + 4 q^8 + 3 q^9 + 16q^{10} +10q^{11} + \textrm{O}(q^{12}),\\
&f_{2,3}(q)=q + 10q^2 + 14q^3 +  7q^4 + 2q^5 + 4 q^6 +   q^7 + 7 q^8 + 2 q^9 + 3 q^{10} +13q^{11} + \textrm{O}(q^{12}),\\
&f_{2,4}(q)=q + 12q^2 +   q^3 + 16q^4 + 5q^5 + 12q^6 + 14q^7 + 16q^8 + 12q^9 + 9 q^{10} +5q^{11} + \textrm{O}(q^{12}).
\end{split}
\end{displaymath}
Let $\rho_{f_{2,i}}:\Gal(\overline{\mathbb{Q}}/\mathbb{Q})\to\GL_2(\mathbb{F}_{17})$ be the mod-17 Galois representation associated to newform $f_{2,i}$ for $1\le i\le 4$. We have the matrices (up to similarity transformation) $\rho_{f_{2,i}}(\textrm{Frob}_p),1\le i\le 4$ are
$$\left[\begin{matrix} 2 & 1 \\ 0 & 2 \end{matrix}\right],
\left[\begin{matrix} 0 & 1 \\ 13 & 3 \end{matrix}\right],
\left[\begin{matrix} 15 & 1 \\ 0 & 15 \end{matrix}\right] \textrm{ and }
\left[\begin{matrix} 0 & 1 \\ 13 & 14\end{matrix}\right]
$$
respectively. So the characteristic polynomial of $\textrm{Frob}_p$ on $J_1(17)$ is
\begin{displaymath}
\begin{split}
P(t)&\equiv (t^2-11t+13)(t^2-4t+4)(t^2-3t-13)(t^2-30t+15^2)(t^2-14t-13) \mod 17\\
&\equiv t^{10} + 6t^9 + 4t^8 + 14t^7 + 8t^6 + 2t^5 + 2t^4 + 3t^3 + 16t^2 + 6t + 13\mod 17.
\end{split}
\end{displaymath}

From which, we have
$$\#X_1(17)(\mathbb{F}_p)\mod 17=(1+6+p)\mod 17=3,$$
and
$$\#J_1(17)(\mathbb{F}_p)\mod 17=P(1)\mod 17=7.$$

\end{example}

\end{document}